\documentclass[review]{elsarticle}

\usepackage{lineno}
\usepackage[colorlinks=true]{hyperref}
\usepackage{amsmath, amssymb, amsthm}
\usepackage{graphicx}
\usepackage{subcaption}
\usepackage[normalem]{ulem}
\usepackage{xcolor}
\usepackage{booktabs}
\usepackage{csvsimple}
\usepackage{longtable}
\usepackage{siunitx}

\usepackage{tikz}
\usepackage{mathabx}
\usepackage{algorithm}
\usepackage{algpseudocode}
\usepackage{bm}
\usepackage[margin=1in]{geometry}

\allowdisplaybreaks

\newtheorem{theorem}{Theorem}

\algtext*{EndIf}
\algtext*{EndFor}
\algtext*{EndFunction}
\algtext*{EndWhile}

\modulolinenumbers[1]

\begin{document}

\begin{frontmatter}

\title{A Branch-and-Price Algorithm for a Team \\Orienteering Problem with Fixed-Wing Drones}

\author{Kaarthik Sundar\corref{cor1}\fnref{fn1}}
\address{Los Alamos National Laboratory, New Mexico 87544, USA}
\cortext[cor1]{Corresponding author}
\fntext[fn1]{Staff Scientist, Information Systems and Modeling, Los Alamos National Laboratory.}
\ead{kaarthik@lanl.gov}

\author{Sujeevraja Sanjeevi\corref{cor2}}
\ead{sujeev.sanjeevi@gmail.com}

\author{Christopher Montez\fnref{fn3}}
\address{Texas A\&M University, College Station, Texas 77843, USA}
\fntext[fn3]{Graduate Student, Department of Mechanical Engineering, Texas A\&M University.}
\ead{yduaskme@tamu.edu}

\begin{abstract}
This paper formulates a team orienteering problem with multiple fixed-wing drones and develops a branch-and-price algorithm to solve the problem to optimality. Fixed-wing drones, unlike rotary drones, have kinematic constraints associated with them, thereby preventing them to make on-the-spot turns and restricting them to a minimum turn radius. This paper presents the implications of these constraints on the drone routing problem formulation and proposes a systematic technique to address them in the context of the team orienteering problem. Furthermore, a novel branch-and-price algorithm with branching techniques specific to the constraints imposed due to fixed-wing drones are proposed. Extensive computational experiments on benchmark instances corroborating the effectiveness of the algorithms are also presented.
\end{abstract}


\end{frontmatter}


\section{Introduction} \label{sec:introduction}
Over the past decade, vehicle routing problems (VRPs) involving drones for delivery \cite{VRPDroneDelivery2016,SidekickTSP2015}, healthcare \cite{DroneHealthcare2017}, monitoring, sensing, mapping, and surveillance have garnered tremendous attention from academia and the industry (see \cite{DroneVRPSurvey2018} for a comprehensive survey of drone VRPs). The focus of this article is to systematically address the kinematic constraints enforced by the drones, which in turn affect the path taken by the drone to go from one point to another, in the context of a team orienteering problem. All the drones currently available in the market can broadly be classified into rotary drones (see Fig. \ref{fig:rotary}) and fixed-wing drones (see Fig. \ref{fig:fixed-wing}). The former is predominantly used in applications pertaining to delivery and healthcare, whereas the latter's use is more common in monitoring, sensing, mapping, and surveillance \cite{DubinsTSP2008,RathinamDubinsTSP2007,GuptaLowerBounds2,GuptaLowerBounds1,FuelTSP2013,MHTSP2017}. The reason for choosing the team orienteering problem to formulate kinematic constraints for drones is that many variants of the classical team orienteering problem with multiple fixed-wing drones and limits on flight times have been directly used in the context of surveillance, sensing, and data collection missions \cite{Pvenivcka2017,Pvenivcka2017a,Tsiogkas2018}. One characteristic of fixed-wing drones that sets it apart from rotary drones with implications to VRPs is the inability to make on-the-spot turns due to kinematic restrictions. This feature of the fixed-wing drones invalidates the assumption that the minimum distance from one point to another is equal to the Euclidean distance between the two points. This is unlike rotary drones where this assumption is closer to reality (see Fig. \ref{fig:paths} for the paths taken by a fixed-wing drone and rotary drone).

\begin{figure}[h]
    \centering
    \includegraphics[width=.5\linewidth]{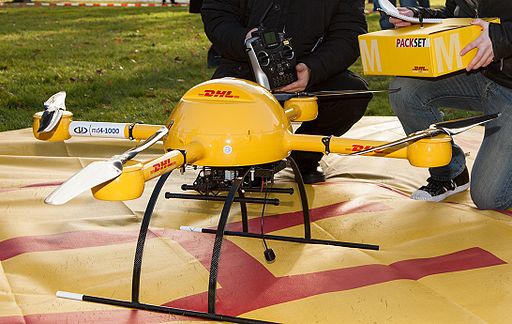}
    \caption{DHL's drone Paketcopter used in package delivery. Source: \url{https://commons.wikimedia.org/wiki/File:Package_copter_microdrones_dhl.jpg}}
    \label{fig:rotary}
\end{figure}

\begin{figure}[ht]
    \centering
    \includegraphics[width=.5\linewidth]{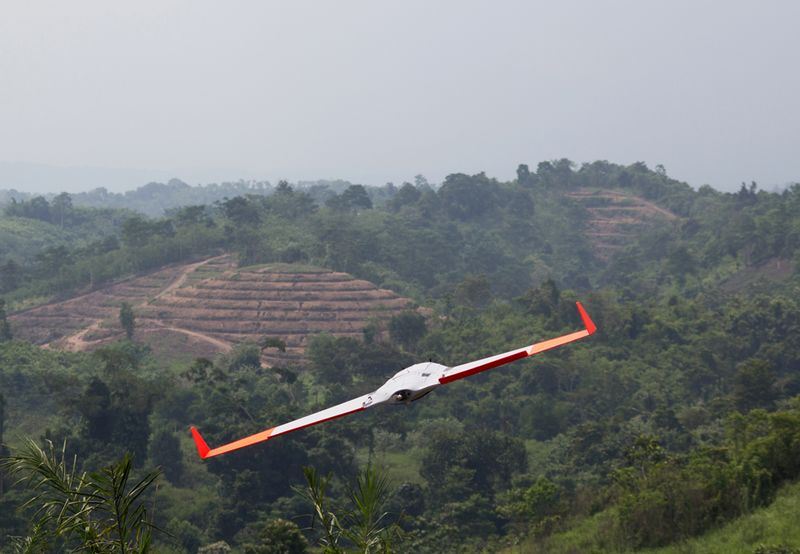}
    \caption{AeroTerrascan's drone Ai450 mapping a field in Indonesia. Source: \url{https://kids.kiddle.co/Image:Agriculture_UAV.jpg}}
    \label{fig:fixed-wing}
\end{figure}

\begin{figure}[ht]
    \centering
    \begin{subfigure}[t]{0.4\textwidth}
        \centering
        \includegraphics[height=1.2in]{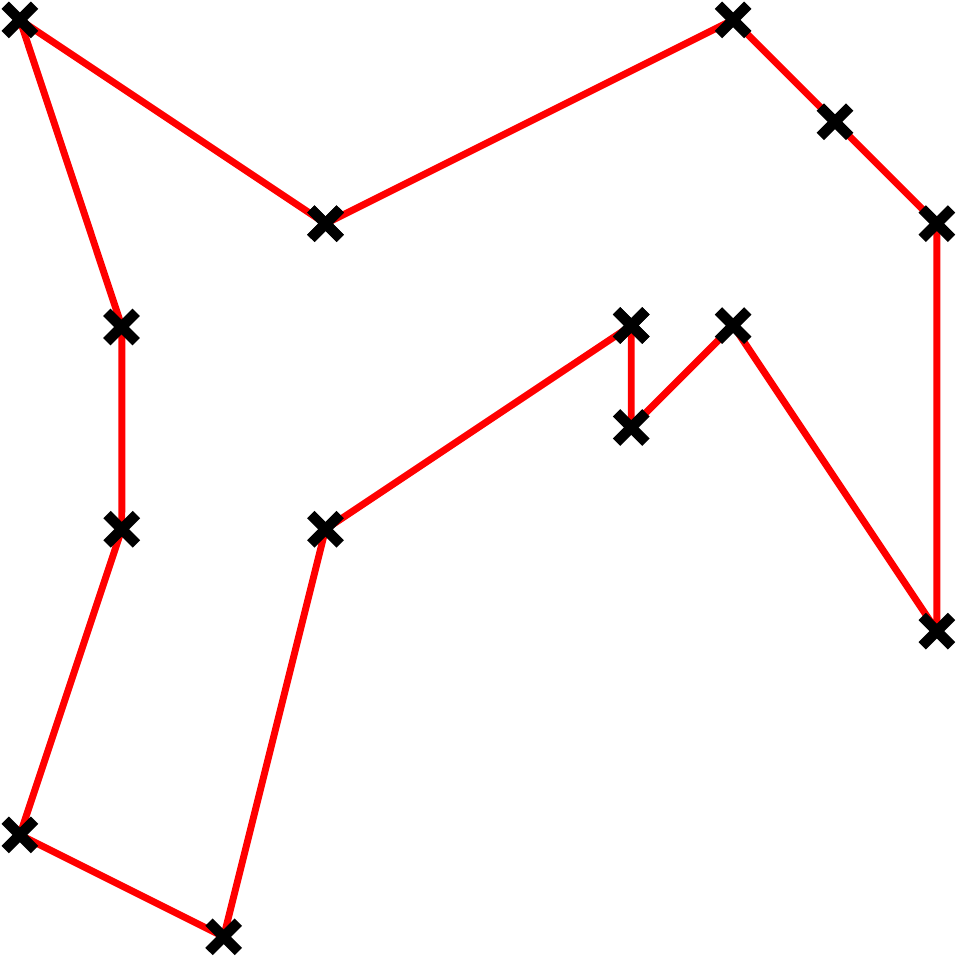}
        \caption{Path taken by a rotary drone. The shortest distance between any two targets for a rotary drone is the Euclidean distance between the targets.}
    \end{subfigure}%
    ~~~~~~
    \begin{subfigure}[t]{0.4\textwidth}
        \centering
        \includegraphics[height=1.2in]{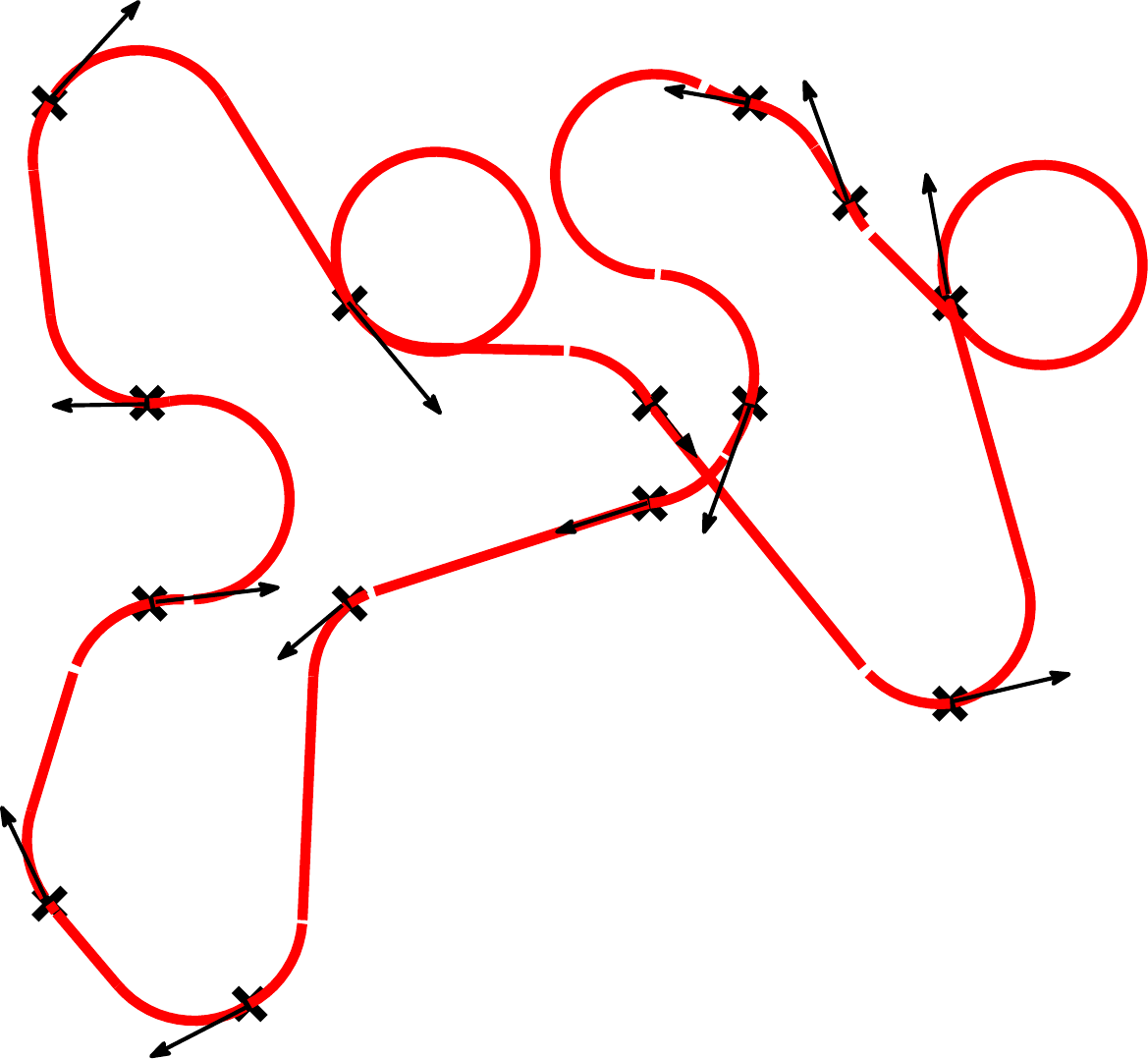}
        \caption{Path taken by a fixed-wing drone. The shortest distance between any two targets for a fixed-wing drone is governed by the drone's kinematic restrictions and its minimum turn radius.}
    \end{subfigure}
    \caption{A comparison of the optimal paths taken by a rotary and fixed-wing drone through a set of targets.}
    \label{fig:paths}
\end{figure}

If we assume that a fixed-wing drone travels at a constant speed $v$, the minimum length path for the drone to travel from a point with Euclidean coordinates $(x_i, y_i)$ to another point $(x_j, y_j)$ would depend on the angle of departure at point $i$ (say $\theta_i$) and angle of arrival at point $j$ (say $\theta_j$). Though the results presented in this article can be extended in a straightforward way to account for points in 3-D, we do not do so for ease of exposition. The kinematic constraints of such a fixed-wing drone is then given by $\dot{x} = v \cos \theta$, $\dot{y}= v \sin \theta$, and $|\dot{\theta}| \leqslant \alpha$, where $\dot x$, $\dot y$, and $\dot \theta$, are the $x$-component of the velocity, the $y$-component of the velocity, and the angular velocity respectively. The value $\alpha$ is referred to as the maximum yaw rate of the drone. Given the kinematic constraints, for any two points $i$ and $j$ with angle of departure $\theta_i$ and arrival $\theta_j$, the shortest path for the drone to start at $i$ and reach $j$ can be computed a priori using the well-known result by Dubins \cite{Dubins}. Given this result, any VRP involving fixed-wing drones requires inclusion of heading angle of vehicles at any given point to be a part of the decision making process to accurately model vehicle paths. This is often ignored while extending most classical VRP algorithms to problems involving drones \cite{DroneVRPSurvey2018}. Recent papers \cite{RathinamDubinsTSP2007,GuptaLowerBounds2,GuptaLowerBounds1,DOP2017} have attempted to address this issue by making heading angles at targets as decision variables and developing heuristic algorithms to solve the resulting problems. However, a comprehensive look at examining its implications in an exact approach like branch-and-price is lacking. This article aims to fill this gap in the literature by formulating a team orienteering problem for homogeneous fixed-wing drones and developing a branch-and-price (B\&P) algorithm to solve the problem to optimality. 

To the best of our knowledge, this is the first attempt at developing an exact branch-and-price algorithm for a team orienteering problem with fixed-wing drones including their kinematic constraints. We expect the algorithms developed in this article to enable modification of algorithms developed for VRP variants without drones to those with drones.

\section{Related work} \label{sec:related-work}
Literature for drone routing is vast and we refer an interested reader to \cite{DroneVRPSurvey2018} for an extensive survey on VRP variants for drones. In this section, we analyze literature specific to drone routing problems taking into account kinematic constraints of drones and the orienteering problem. For ease of exposition, we will classify literature using the following two categories: (i) the orienteering problem, its variants, and algorithms and (ii) drone VRPs that enforce kinematic constraints on the drones. 

The orienteering problem (OP) was first introduced in \cite{Chao1996b} in the context of the sport of orienteering where competitors start at a control point, try to visit as many checkpoints as possible and return to the control point within a given time frame. Each checkpoint has a certain score and the objective is to maximize the total collected score. Ever since the problem's inception, many variants of the OP have been used in mapping applications \cite{DOP2017}, tourism, logistics etc. (see \cite{Vansteenwegen2011,Gunawan2016} and references therein). Both exact \cite{Fischetti1998} and heuristic approaches \cite{Vansteenwegen2009,Sevkli2006,Souffriau2010,Ramesh1991,Golden1988} to solve OP have received extensive attention in the literature. The focus of this article is on the team orienteering problem (TOP) \cite{Chao1996a}, a multi-vehicle generalization of OP that corresponds to playing the game of orienteering by teams of several persons with each collecting scores during the same time span. Practical applications of the TOP range from athlete recruitment \cite{Butt1994} to technician routing \cite{Tang2005}. The main reason for choosing the TOP for this article is its applicability in a mapping or surveying application for multiple fixed-wing drones. Though many other variants of the OP could also have been considered, we chose the TOP for its simplicity and to demonstrate computational issues encountered when including only kinematic constraints of fixed-wing drones without adding other complicating mission restrictions (like time-windows \cite{Vansteenwegen2009-tw}).  Another variant of the OP that is very closely related to the problem considered in this paper is the Set Orienteering Problem (SOP) \cite{Archetti2018} which aims to solve an OP with a pre-specified number of sets of targets and the goal is to visit at most one target from each set; to the best of our knowledge, no TOP variants of the SOP have been addressed in the literature. Many exact approaches have been developed for the TOP including branch-and-cut \cite{El2016,Bianchessi2018}, branch-and-price \cite{Boussier2007,Keshtkaran2016}, and branch-cut-and-price algorithms \cite{Poggi2010}. Among all these approaches, the approaches in \cite{Bianchessi2018} and \cite{Keshtkaran2016} are known to be the best and the second best algorithms, respectively, to solve the TOP to optimality. Furthermore, branch-and-price has also been used to solve certain variants of the VRP with drones (without kinematic constraints) successfully to optimality \cite{Wang2019}. Hence, in this article, we develop a branch-and-price algorithm to solve the TOP for fixed-wing drones to optimality with a focus on algorithmic aspects specific to these drones. In this context, we also note that the classical TOP is a relaxation of the TOP for fixed-wing drones in the sense that it ignores kinematic constraints. Hence, any exact approach developed for the TOP would directly provide an upper bound to the optimal objective of the TOP with fixed-wing drones. Next, we discuss literature in the context of drone VRPs which have addressed some version of the kinematic constraints and developed heuristic algorithms to solve them.

The first work in the literature to stress the importance of kinematic constraints in the context of path planning for fixed-wing drones is \cite{Tang2005a}. Though no explicit VRPs was formulated in that \cite{Tang2005a}, it was the starting point for many papers \cite{DubinsTSP2008,RathinamDubinsTSP2007,GuptaLowerBounds2} that formulated the traveling salesman problem (TSP) with a fixed-wing drone. This problem is also referred to as the ``Dubins TSP'' in the literature, since it was L. E. Dubins in his seminal paper in 1957 \cite{Dubins} who solved the shortest path problem for a fixed-wing drone to go from source to a destination with specified angles of departure and arrival while satisfying the kinematic constraints of the drone.  The focus of all these articles \cite{DubinsTSP2008,RathinamDubinsTSP2007,GuptaLowerBounds2} was to develop a technique to include the kinematic constraints of the fixed-wing drones to the TSP rather than solve the resulting problems themselves to optimality. The approach that was taken by all the papers was to eventually transform the TSP with fixed-wing drones to an asymmetric TSP, albeit a huge one, and solve it either using the Concorde (\url{http://www.math.uwaterloo.ca/tsp/concorde.html}) TSP solver or heuristics. This severely restricted the problem sizes that can be solved to optimality. A similar approach using the Variable Neighborhood Search (VNS) algorithm was developed for the Dubins OP with a single fixed-wing drone \cite{DOP2017}. In \cite{DOP2017}, the focus was again to formulate the OP to include kinematic constraints and use a VNS to solve the problem heuristically. Finally, all approaches and problems considered thus far in this section that includes kinematic constraints only deal with single vehicle variants. To the best of our knowledge, there is no work in the literature that attempts to develop a comprehensive exact approach to solve multi-drone VRPs with kinematic constraints to optimality. This article is the first work that takes a step in that direction using a TOP approach. 

\subsection{Contributions} \label{subsec:contributions}
In summary, the following are the main contributions of the article: (i) we formulate a team orienteering problem for homogeneous fixed-wing drones and develop the first \textit{concurrent} multi-threaded B\&P algorithm to solve it to optimality (ii) we develop a `Decremental State Space Relaxation' (DSSR) to solve the pricing problem while utilizing the structure of the problem; this approach is based on the state-of-the-art DSSR algorithm \cite{Righini2009} that is used in the literature to solve pricing problems occurring in generic VRPs, (iii)  problem-specific branching strategies to further speed up the pricing problem solution approach are presented, and finally (iv) extensive computational experiments that corroborate the effectiveness of the concurrent B\&P to solve the problem and show the efficacy of utilizing the DSSR are detailed.

The rest of the article is organized as follows: in Sec. \ref{sec:statement} we present the formal problem statement after introducing suitable notations, in Sec. \ref{sec:formulation}, we present the mathematical formulation for the problem, followed by the B\&P algorithm in Sec. \ref{sec:bp} and computational results in Sec. \ref{sec:results}. Finally, Sec. \ref{sec:conclusions} details the potential avenues for future work and concludes the article. 

\section{Problem statement} \label{sec:statement}
Throughout the rest of the article, we shall refer to the TOP with fixed-wing drones as the ``Dubins Team Orienteering Problem'' (DTOP). DTOP is a generalization of the TOP with multiple homogeneous fixed-wing drones whose paths have to satisfy kinematic constraints. In practice, this involves accounting for heading angles at targets since the shortest path between any pair of targets for any vehicle depends on the heading angle of the vehicle at both targets and its maximum yaw-rate. If the heading angles at each target for any vehicle is specified a priori, then the DTOP reduces to the asymmetric TOP. The asymmetry arises from the fact that the shortest path length may change when the vehicle is traveling from target $t$ to $u$ as opposed to $u$ to $t$ even for fixed-heading angles at $t$ and $u$. Suppose that the sequence of target visits for each path is specified. Then, computing a path satisfying kinematic constraints through this sequence involves computing the heading angles at each target. This itself is an NP-hard optimal control problem with intermediate state constraints \cite{Kaya2017}. In this article, we present an approach that decouples the combinatorial and optimal control problems and reduces the DTOP to a pure combinatorial problem by discretizing heading angles at each target. This approach is not new and has been previously proposed in the literature \cite{GuptaLowerBounds1} in the context of developing heuristics and transformation algorithms. Once decoupled, we obtain a generalization of the asymmetric TOP with a set of vertices for each target with the vehicles having to visit at most one vertex from each target set. We develop an exact B\&P algorithm for this discretized version of the DTOP. We note that this discretization scheme is very general and can be used for any VRP with fixed-wing drones. Throughout the rest of the article, we refer to the discretized version of the DTOP as the D-DTOP. In this context, we remark that the focus of the article is to solve the D-DTOP to optimality for a fixed number of discretizations rather than solving the DTOP. 

We first introduce some notations to formally state the D-DTOP. We are given $m$ identical fixed-wing drones or Dubins vehicles with a maximum yaw-rate of $\alpha$. All vehicles are assumed to travel at constant speed. Let the set of targets be denoted by $T = \{1, \dots, n\} \cup \{s, d\}$ ($s$ is the source target where $m$ vehicles are initially stationed and $d$ is the destination where the $m$ vehicle paths have to terminate). Associated with each target $t \in T$ is a non-negative score $p_t$ that is collected when any vehicle visits $t$. The targets $s$ and $d$ are assigned zero scores. Any vehicle can arrive at and depart from any target at a heading angle chosen from the set $\Theta = \{\theta_1, \theta_2, \dots, \theta_k\}$. Hence, each target $t \in T$ is associated with a set of $k$ vertices denoted by the set $V^t$. When any vehicle visits a vertex $v \in V^t$, this in turn translates to the vehicle arriving at and departing from the target $t$ at a heading angle that corresponds to the vertex $v$. For any vertex $v$, we let $\beta_v$ denote the heading angle corresponding to the vertex $v$. A path from vertex $p \in V^t$ to $q \in V^u$ for distinct targets $t,u \in T$ is assigned a length $c_{pq}$, that is given by the shortest Dubins path from target $t$ to target $u$ with angle of departure and arrival set to $\beta_p$ and $\beta_q$. With these notations, the D-DTOP is formulated on a directed graph $G = (V, E)$, where $V = \bigcup_{t \in T} V^t$ is the union of the vertex sets for all targets. The edge set $E$ consists of all the edges between any pair of vertices $i, j \in V$ that connect distinct targets. The objective of the D-DTOP is to compute $m$ paths, one for each vehicle, that start at some vertex in the source target, visits a subset of vertices such that at most one vertex is visited from each set $V^t$, $t \in T$ and ends in some vertex in the destination target, while keeping the length of each path less than a pre-specified limit $L_{\max}$. Similar to the TOP, the D-DTOP aims to maximize the sum of collected scores. A feasible solution to an instance of the D-DTOP is shown in Fig. \ref{fig:illustration}. In the next section, we present a set-packing formulation for the D-DTOP inspired by previous work on the TOP \cite{Boussier2007}. 

\begin{figure}
    \centering
    \includegraphics[scale=0.7]{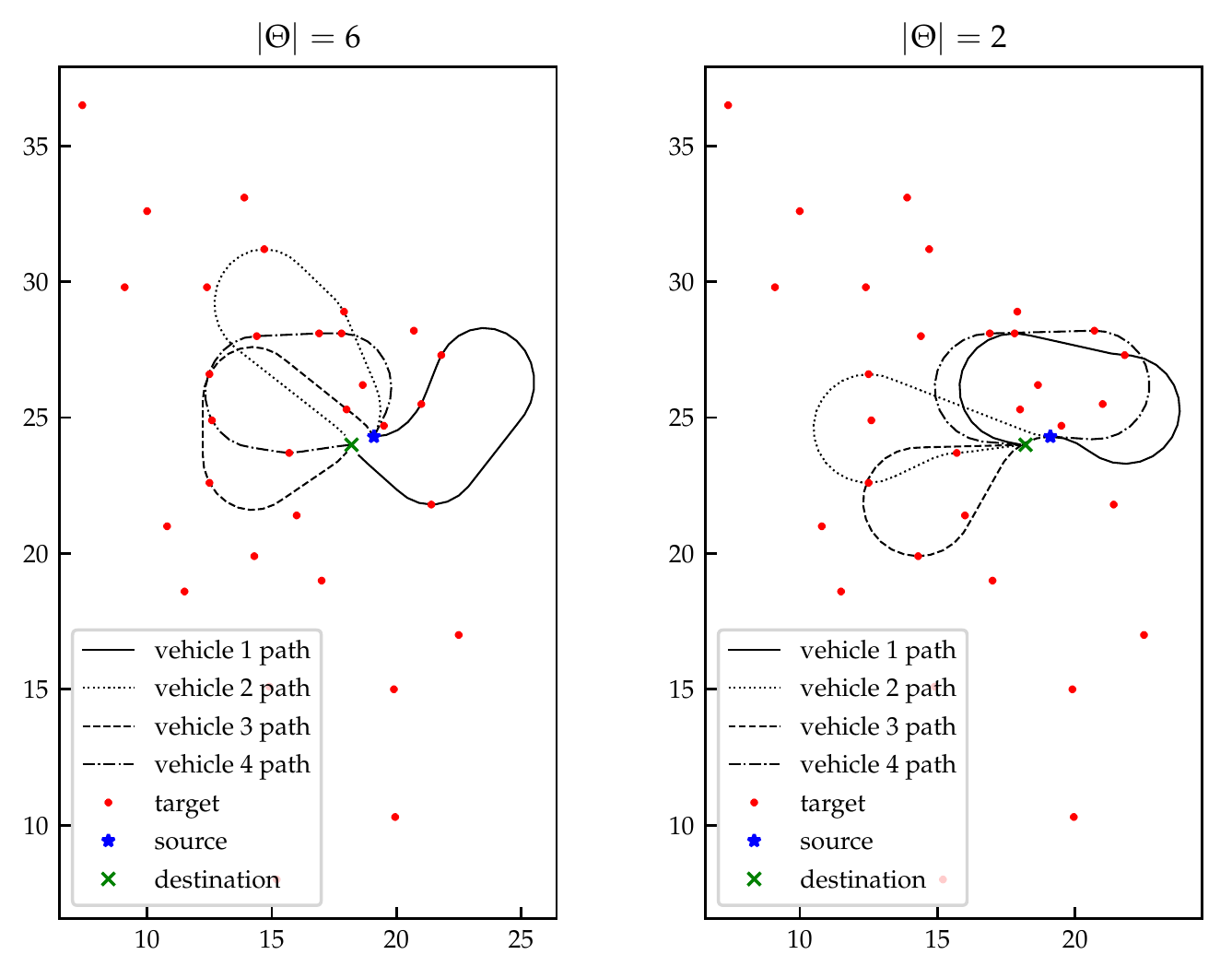}
    \caption{Feasible solutions for an instance of the D-DTOP for 6 and 2 discretizations.}
    \label{fig:illustration}
\end{figure}

\section{Mathematical formulation} \label{sec:formulation}
Let $R = \{r_1, r_2, \dots, r_{|R|}\}$ denote the set of possible routes where each route starts at some vertex in $V^s$, visits a subset of vertices such that at most one vertex is visited from each target, and ends at some vertex in $V^d$ with total path length at most $L_{\max}$. Let $z_r$ be a binary decision variable that takes a value of $1$ if route $r \in R$ is used and $0$ otherwise. Let $p_r$ denote the route score, i.e. the sum of scores of targets on the route. We identify targets visited by route $r$ using a binary parameter $a_{tr}$ that has value $1$ for each target $t$ visited by the route and $0$ for other targets. Then the D-DTOP can be formulated as follows:
\begin{flalign}
(\mathcal F) \qquad & \max \quad \sum_{r \in R} p_r z_r & \label{eq:obj} \\ 
\text{subject to:} \qquad & \sum_{r \in R} z_r \leqslant m, & \label{eq:route-cover} \\ 
\qquad & \sum_{r \in R} a_{tr} z_r \leqslant 1, \qquad \forall\  t \in T\setminus\{s, d\}, & \label{eq:target-cover} \\ 
\qquad & z_r \in \{0, 1\} \qquad \forall\ r \in R. \label{eq:binary} 
\end{flalign}
Constraint \eqref{eq:route-cover} limits the number of routes to $m$. Constraints \eqref{eq:target-cover} ensure that at most one visit is made to each target. We first present an approach to solve the continuous relaxation of Eq. \eqref{eq:obj} -- \eqref{eq:binary} using column generation, a natural fit here due to the exponential size of $R$. We then embed this approach into a branch-and-bound framework to find an optimal solution to the binary formulation. Throughout the rest of the article, we will refer to the linear relaxation of Eq. \eqref{eq:obj} -- \eqref{eq:binary} as the Master Problem (MP).

\section{Branch-and-Price algorithm} \label{sec:bp}
Our proposed B\&P algorithm for the exact resolution of the D-DTOP is structured similar to the B\&P algorithm for the TOP in \cite{Boussier2007}. Our approach deviates from \cite{Boussier2007} in the algorithms used for solving the pricing sub-problems, the branching scheme and other enhancements specific to the D-DTOP. In the subsequent paragraph, we present the column generation algorithm that computes an upper bound for the D-DTOP by solving MP.

\subsection{Column-generation for solving MP} \label{subsec:cg}
The algorithm starts with a Restricted MP (RMP) that contains a limited number of routes in $R$. It then iterates between solving RMP to update reduced cost values and searching for positive reduced cost routes. It terminates when no such route can be found as optimality has been reached. To define the reduced cost of a route, we let $\lambda_0 \geqslant 0$ and $\lambda_t \geqslant 0$ denote the dual variables associated with the constraints in \eqref{eq:route-cover} and \eqref{eq:target-cover} respectively. A route $r \in R$ has a positive reduced cost if
\begin{flalign}
\lambda_0 + \sum_{t \in T} a_{tr} \lambda_t < p_r \quad \text{or} \quad p_r - \sum_{t \in T} a_{tr} \lambda_t > \lambda_0. \label{eq:prc-route}
\end{flalign}
Hence, finding a route with a positive reduced cost is equivalent to solving a resource-constrained elementary shortest path problem (RCESPP). We note that the RCESPP itself is an NP-hard problem \cite{Desrosiers2005}.

\subsection{Pricing problem algorithm} \label{subsec:pricing}
Our approach to solve the pricing problem builds on the bounded bi-directional dynamic programming procedure with a \textit{Decremental State Space Relaxation} (DSSR). DSSR was originally proposed in \cite{Righini2008}, and also shown to be computationally effective for TOP and its variants \cite{Righini2009,Keshtkaran2016}. We directly use this procedure on the graph $G$ with the updated reduced costs in Sec. \ref{subsec:cg} combined with special branching rules (detailed in the subsequent sections) that are unique to the D-DTOP. A flow chart of the pricing algorithm is shown in Fig. \ref{fig:pricing} for clarity.

\begin{figure}
    \centering
    \includegraphics[scale=0.7]{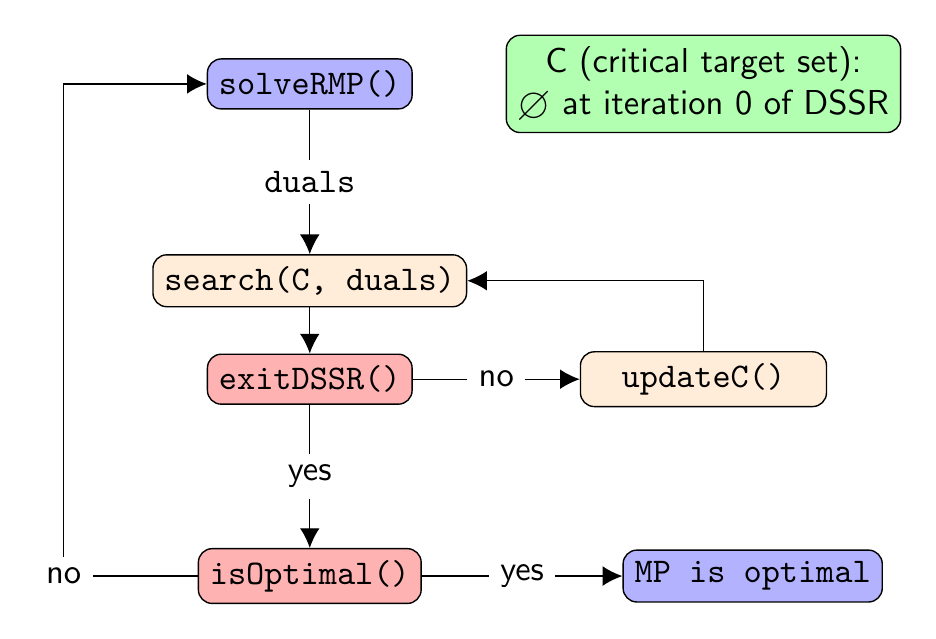}
    \caption{Flow chart of the pricing algorithm with the search algorithm implemented using a DSSR acceleration scheme.}
    \label{fig:pricing}
\end{figure}

We first present a brief overview of the labeling procedure proposed in \cite{Righini2009} that relies on DSSR. It uses bi-directional search by building and extending forward and backward labels (partial paths). Paths are generated by feasible joins of forward labels with backward labels.  A critical resource with monotonic consumption along paths is selected, and the exploration in each direction is stopped when half of such critical resource has been consumed. The search procedure of \cite{Righini2009} performs these operations sequentially; all forward extensions are generated, followed by backward extensions and then joins. The DSSR is an effective acceleration technique proposed by \cite{Righini2008}. In this method, the elementarity condition is only checked for a subset of vertices called ``critical vertices'' and multiple visits are allowed to vertices that are not critical. Whenever the solution to this relaxed problem contains vertices visited more than once, either a subset or all such vertices are added to the critical set, and the problem is solved again. One way to use this procedure for multi-vehicle problems is shown in \cite{Righini2009,Keshtkaran2016}. This version includes exit conditions and label dominance rule relaxations in early iterations of the DSSR. 

We now describe the main details of the pricing problem algorithm implemented for the D-DTOP. For the D-DTOP, though the DSSR builds paths through vertices, it relies on the notion of \textit{critical targets}, a set of targets to which multiple visits are not allowed. It is updated at the end of each search iteration with targets that are visited multiple times by the highest reduced cost path. Note that marking a target as critical prevents multiple visits to all vertices of the target. Finally, the path length is used as a bounding resource for the D-DTOP since this is the only resource. Forward and backward labels are extended until their partial path lengths do not exceed $L_{\max}/2$. The DSSR loop is started with an empty set of critical targets (see Fig. \ref{fig:pricing}). In each iteration, the label extension phase in both the directions is followed by joining the forward and backward labels using the edge connecting the two labels. If the length of the resulting path is greater than $L_{\max}$, the path is discarded. The join phase keeps track of the highest reduced cost path (possibly non-elementary) and all elementary positive reduced cost paths. When the number of elementary positive reduced cost paths exceeds $500$, the join phase is interrupted. This parameter is referred to as $\texttt{MaxPath}$ in \cite{Keshtkaran2016}. If in the current DSSR iteration, the highest reduced path is non-elementary and the join phase has produced elementary positive reduced cost paths, we add these paths to the RMP and solve the RMP to get updated dual values. If no elementary paths were found in a DSSR iteration and the highest reduced cost path was non-elementary, then all the targets visited more than once in this path are added to the set of critical targets and the bounded bidirectional search is restarted. The process is iterated until we do not find any path with positive reduced cost. In the next section, we present dominance rules that aid in speeding up the label extension phase of each DSSR iteration.

\subsection{Dominance rules} \label{subsec:dominance}
Before we present the dominance rules, we introduce some notations for ease of presentation. A label is defined by $(S, \ell, c, i)$, where $S$ is the set of critical targets visited, $\ell$ is the length of the partial path, $c$ is the reduced cost of the partial path, and $i$ is the last vertex reached. Given two labels, $L_1 = (S_1, \ell_1, c_1, i)$ and $L_2= (S_2, \ell_2, c_2, i)$, we say $L_1$ dominates $L_2$ if the following conditions are satisfied:
\begin{flalign}
\text{(i) } c_1 \geqslant c_2 \text{ (ii) } \ell_1 \leqslant \ell_2, \text{ and (iii) } S_1 \subseteq S_2 \label{eq:dominance}
\end{flalign}
with at least one of the inequalities being strict. Given this definition of dominance, we discard (a) a newly generated label if it is dominated by an existing label and (b) an existing label if it is dominated by the newly generated label.

Furthermore, we also discard labels associated with paths containing cycles of length two on targets (referred to as two-cycle elimination). This is easily done by not extending a given label to the target of its predecessor vertex. When we eliminate labels to avoid two-cycles, care must be taken to differentiate between dominance and discarding. Specifically, if a label $L$ dominates a label $L'$, we can discard $L'$ only if one of the following three conditions are satisfied \cite{Irnich2005}:
\begin{itemize}
    \item $L$ has the same predecessor as $L'$, or
    \item $L$ cannot be extended to the predecessor target on its own partial path, or
    \item There exists a label other than $L$ that dominates $L'$ and has a different predecessor target than $L$.
\end{itemize}
In our implementation of the DSSR, we utilize all the aforementioned conditions to discard labels. In the next section, we present a branching scheme that is specific to the problem structure induced by fixed-wing drones.

\subsection{Branching scheme} \label{subsec:branching}
If an optimal solution to the MP in Eq. \eqref{eq:obj}--\eqref{eq:binary} is fractional, we use a branching scheme similar to the vertex branching scheme in \cite{Boussier2007} for the TOP. We perform the following types of branching in stated order:

\begin{enumerate}[(i)]
    \item \textit{Target branching}: Select a target $t$ with fractional flow. Create two sub-problems, one enforcing a visit to $t$ and one forbidding visits to $t$. When multiple targets are available, select one with the least value of $\lambda_t - p_t$.
    \item \textit{Target connection branching}: Select a target connection $(t_1, t_2)$ with fractional flow. If visits to $t_1$ or $t_2$ are already enforced, create two sub-problems by enforcing and prohibiting direct connections between $t_1$ and $t_2$. Otherwise, create three sub-problems. The first prohibits visits to $t_1$. The second enforces visits to $t_1$ and prohibits direct connections from $t_1$ to $t_2$. The third enforces visits to $t_1$ and enforces a direct connection between $t_1$ and $t_2$. When multiple candidate edges are available, select an edge that starts from a target with the least value of $\lambda_t - p_t$.
\end{enumerate}

We forbid targets and target connections by building a reduced graph. This graph is constructed from the original graph by removing vertices of all forbidden targets, and edges between vertices of all forbidden target connections. Target and connection visits are enforced by adding additional constraints to MP. Presenting these constraints requires more notation. Let $\tilde{T}$ be the set of targets in the reduced graph, $ET \subseteq \tilde{T}$ be the set of targets with enforced visits, and $C = \{(t_1, t_2) : t_1 \in \tilde{T}, t_2 \in \tilde{T}\}$ be the set of enforced target connections. Let $b_{cr}$ be a binary parameter with value $1$ if route $r$ uses target connection $c = (t_1, t_2) \in C$, i.e. uses an edge between a vertex in $V^{t_1}$ and $V^{t_2}$, and $0$ otherwise. Solutions at each node are found using the following modified formulation of the MP:

\begin{flalign}
(\mathcal F_1) \qquad & \max \quad \sum_{r \in R} p_r z_r - M y & \label{eq:obj-y} \\ 
\text{subject to:} \qquad & \sum_{r \in R} z_r \leqslant m, & \label{eq:route-cover-y} \\ 
\qquad & \sum_{r \in R} a_{tr}  z_r \leqslant 1, \qquad \forall t \in \tilde{T}, & \label{eq:target-cover-y} \\ 
\qquad & \sum_{r \in R} a_{tr} z_r + y \geqslant 1, \qquad \forall t \in ET, &\label{eq:target-enforced-y} \\
\qquad & \sum_{r \in R} b_{cr} z_r + y \geqslant 1, \qquad \forall c \in C, &\label{eq:target-edge-enforced-y} \\
\qquad &  y \geqslant 0, z_r \in \{0, 1\} \qquad \forall r \in R. \label{eq:binary-y} 
\end{flalign}
Visits are enforced by the new constraints \eqref{eq:target-enforced-y} and \eqref{eq:target-edge-enforced-y}. The dual values of these additional constraints are accounted for in the reduced costs appropriately. As the branching scheme can cause $\mathcal{F}_1$ to be infeasible, we detect it with an artificial non-negative variable $y$ that carries a large negative coefficient $(-M)$ in the objective. If the solution to $\mathcal{F}_1$ has a non-zero value for $y$, the corresponding node can be pruned by infeasibility. The forthcoming theorem proves that the target branching and target connection branching defined above is sufficient for the B\&P algorithm to obtain an optimal solution to the D-DTOP i.e., it proves the exhaustivness of our branching scheme. 

\begin{theorem} \label{thm:main-theorem}
If a solution to the linear relaxation of the MP has integral flows into every target and fractional flows into one or more vertices, then there exists an integral solution to the MP with the same objective value. 
\end{theorem}
\begin{proof}
First, it is not difficult to see that if a solution to the linear relaxation of the MP has integral flows into every target, then the flow between any pair of targets is integral and that any fractional path that visits a sequence of targets would have one or more fractional paths that visit the same targets in exactly the same sequence. All these paths have the same objective value since the profits associated with any vertex in a particular target is the same. Furthermore, the pricing algorithm in Sec. \ref{subsec:pricing} guarantees that every path (column) generated and added to the RMP satisfies the budget constraints. Hence, choosing one arbitrary path for each target visit sequence from the set of fractional paths that correspond to the same target visit sequence would result an integral solution to the MP with the same objective value.
\end{proof}

The above theorem indicates that if the solution to the linear relaxation of the MP has fractional flows into one or more vertices and integral flows into one or more targets, then it can be pruned by optimality. This makes the branching scheme exhaustive and precludes the need to branch on vertex or vertex connection visits.

\section{A note on concurrent implementation}  \label{sec:concurrency}
The branching scheme presented in the previous section along with the column generation procedure in Sec. \ref{subsec:cg} is used to implement a multi-threaded concurrent B\&P algorithm to solve the D-DTOP. In this context, this is achieved by processing nodes of the branch-and-bound tree concurrently. We implemented this concurrency using coroutines \cite{Moura2009} in a Communicating Sequential Processes (CSP) framework \cite{Hoare2002}. 

In this article, all the implementations of the branch-and-price algorithm are concurrent. Later in the computational results section, we present the gain in run-times when moving from single-threaded concurrent implementation to multi-threaded concurrent implementation for the proposed algorithm. We also remark that it is a common abuse of nomenclature to refer to multi-threaded concurrent implementations as parallel implementation due to the fact that in a multi-threaded implementation tasks are still being run simultaneously. But it is worthwhile to understand the subtle difference that in a multi-threaded concurrent implementation each thread can actually work with tasks from different functions whereas in a multi-threaded parallel implementation, the threads work only on the function that they are asked to process. The implementation of both the single-threaded and multi-threaded concurrent versions of the algorithm has been open-sourced and is available at \url{https://github.com/sujeevraja/fixed-wing-drone-orienteering}.

\section{Computational results} \label{sec:results}
In this section, we discuss the computational results of the branch-and-price algorithm. The algorithm was implemented in the Kotlin programming language with CPLEX 12.7 as the LP solver. All experiments were performed on an Intel Broadwell E5-2695 processor with a base clock rate of 2.10 GHz, a RAM of 128 GB and with 36 cores. All computation times reported are expressed in seconds. We imposed a time limit of $1$ hour for each run of the algorithm. Furthermore, to track visits to critical targets, we found that bit-based operations were computationally more efficient than using boolean arrays. For all runs, the values of $M$ in Eq. \eqref{eq:obj-y} was set to $10^5$.

\subsection{Instance generation} \label{subsec:instance-info}
The performance of the algorithm was tested using the standard benchmark library for the TOP (\url{https://www.mech.kuleuven.be/en/cib/op#section-3}) on instances with a maximum of $66$ targets. The total number of instances that satisfy this criteria is $267$. For each of these instances, three D-DTOP variants were generated with $(2,4,6)$ discretizations of heading angles. Possible heading angles for an instance at each target were obtained by uniformly discretizing the set $[0, 2\pi)$. For every vehicle, the turn radius was set to $1$ unit and the length of the path from vertex $p \in V^t$ to $q \in V^u$ for distinct targets $t,u \in T$ is computed as the shortest Dubins path from target $t$ to target $u$ with angle of departure and arrival set to $\beta_p$ and $\beta_q$, respectively. In summary, the number of instances of for the D-DTOP totals to $801$. 

\subsection{Performance of the branch-and-price algorithm} \label{subsec:bp-performance}
The first set of computational experiments were designed to evaluate the performance of the multi-threaded concurrent B\&P algorithm. The multi-threaded B\&P algorithm can process at most $8$ branch-and-bounds nodes simultaneously.  Table \ref{tab:instance-info} presents the number of instances that were optimal, and the number of them that timed-out respectively. 
\begin{table}
\centering
\caption{`opt' and `TO' in the table header represents the number of instances that were optimal and timed-out within a computational time limit of 1 hour respectively.} 
\label{tab:instance-info}
\begin{filecontents*}{tables/exhaustive_counts.csv}
num_targets,opt_2,to_2,opt_4,to_4,opt_6,to_6
21,33,0,33,0,33,0
32,54,0,54,0,52,2
33,59,1,51,9,47,13
64,42,0,40,2,32,10
66,72,6,62,16,57,21
\end{filecontents*}
\csvreader[
    before reading=\scriptsize,
    after reading=\normalsize,
    tabular=|c|cc|cc|cc|,
    table head=\hline
    $|T|$ & \multicolumn{2}{c|}{$|\Theta|=2$} & \multicolumn{2}{c|}{$|\Theta|=4$} & \multicolumn{2}{c|}{$|\Theta|=6$} \\ 
    \hline
    & opt & TO & opt & TO & opt & TO \\,
    late after line=\\,
    late after last line=\\\hline,
    ]{tables/exhaustive_counts.csv}{1=\t,2=\otwo,3=\totwo,4=\ofour,5=\tofour,6=\osix,7=\tosix}%
    {\t & \otwo & \totwo & \ofour & \tofour & \osix & \tosix}
\end{table}
$80$ out of the total $801$ instances timed-out and the remaining instances were solved to optimality. It is also important to note that $15$ instances that were optimally solved had an objective value of $0$, i.e., no targets were visited by any of the vehicles in these instances. This illustrates the value of modeling kinematic constraints when formulating VRPs involving fixed-wing drones: a path that visits a subset of targets when not taking into account the kinematic constraints can potentially be infeasible when taking into account the kinematic constraints of the drone.

\paragraph{Table column names and description} We now present the exhaustive results of all runs that were optimal for at least one of the discretization values. Tables \ref{tab:21} -- \ref{tab:66} present these results for different numbers of targets. Nomenclature used in these tables is as follows: \textbf{n}: instance name, \textbf{rub}: LP relaxation value at the root node of the branch-and-bound tree, \textbf{opt}: optimal objective value, \textbf{nn}: number of nodes explored in the branch-and-bound tree, and \textbf{time}: computation time in seconds. 

In Tables \ref{tab:21} -- \ref{tab:66}, the instances for which the `\textbf{time}' column contains a value of $3600.00$ timed-out. For all such instances, the column `\textbf{opt}' is the objective value of the best feasible solution obtained. Also, whenever an instance timed-out, if it has a `--' in the `\textbf{rub}' column, it implies that the root relaxation failed to solve to optimality within the time limit. 

\paragraph{Effect of increasing heading angle discretization} A trend revealed from Tables \ref{tab:21} -- \ref{tab:66} is that as the number of heading angle discretizations is increased, the number of instances with an optimal objective value of $0$ decreases. This aligns with the intuition that increasing the heading angle discretizations at each target increases the search space thereby increasing the chance of finding a feasible solution that visits a subset of targets to the D-DTOP. Furthermore, the run-times also increase with increasing discretization. The reason for this behaviour is that with larger heading angle discretization, the cardinality of the vertex set $V^t$, associated with each target $t \in T$, increases with increasing $|\Theta|$. This in turn increases the size of the instance that is being solved. Finally, the optimal objective value of the D-DTOP is also observed to increase with increasing $|\Theta|$. This again is due to the fact that the solution space of the problem for greater $|\Theta|$ is larger. In fact, for a sequence of heading angle discretizations $\{\Theta_1, \Theta_2, \cdots, \Theta_k\}$ such that $\Theta_1 \subset \Theta_2 \subset \cdots \subset \Theta_k$, it is easy to see that the objective value for the corresponding D-DTOPs will increase monotonically and will bounded above by the objective to the TOP problem \cite{GuptaLowerBounds1,GuptaLowerBounds2}. Intuitively, this is a natural result because of the fact that the TOP is a relaxation of the D-DTOP. In all the results, this observation holds for when $|\Theta| \in \{2, 4\}$ and $|\Theta| \in \{2, 6\}$. The discretization angles for $|\Theta|=2$, $|\Theta|=4$, and $|\Theta|=6$ are $\{0, \pi\}$, $\{0, \pi/2, \pi, 3\pi/2\}$, and $\{0, \pi/3, 2\pi/3, \pi, 4\pi/3, 5\pi/3 \}$, respectively. Hence, the monotonicity of the objective is guaranteed when we go from $|\Theta| = 2$ to $|\Theta| = 4$ and from $|\Theta| = 2$ to $|\Theta| = 6$.

\setlength\tabcolsep{4pt}
 \begin{filecontents*}{tables/full_21.csv}
instance_name,rub_2,opt_2,nodes_2,time_2,rub_4,opt_4,nodes_4,time_4,rub_6,opt_6,nodes_6,time_6
p2.2.a,50.00,50.00,1,0.15,50.00,50.00,1,0.15,60.00,60.00,1,0.16
p2.2.b,80.00,80.00,1,0.16,85.00,85.00,1,0.17,90.00,90.00,1,0.20
p2.2.c,80.00,80.00,1,0.17,85.00,85.00,1,0.19,95.00,95.00,1,0.23
p2.2.d,90.00,90.00,1,0.17,100.00,100.00,1,0.21,115.00,110.00,7,0.68
p2.2.e,100.00,100.00,1,0.18,120.00,115.00,7,0.80,135.71,135.00,3,0.81
p2.2.f,115.00,115.00,1,0.21,152.50,145.00,7,0.78,181.25,180.00,3,0.90
p2.2.g,125.00,125.00,1,0.21,167.50,165.00,3,0.85,190.00,190.00,1,1.10
p2.2.h,140.00,140.00,1,0.23,186.67,180.00,17,2.28,200.00,200.00,1,0.43
p2.2.i,160.00,160.00,1,0.30,205.00,200.00,5,2.32,230.00,230.00,1,1.06
p2.2.j,163.33,160.00,25,1.72,220.00,220.00,1,0.72,230.00,230.00,1,2.05
p2.2.k,190.00,190.00,1,0.32,236.67,230.00,5,2.93,260.00,260.00,1,2.15
p2.3.b,15.00,15.00,1,0.15,25.00,25.00,1,0.14,35.00,35.00,1,0.16
p2.3.c,50.00,50.00,1,0.14,60.00,60.00,1,0.15,70.00,70.00,1,0.16
p2.3.d,50.00,50.00,1,0.15,60.00,60.00,1,0.17,90.00,90.00,1,0.16
p2.3.e,50.00,50.00,1,0.14,80.00,80.00,1,0.16,105.00,105.00,1,0.18
p2.3.f,95.00,95.00,1,0.15,105.00,105.00,1,0.17,105.00,105.00,1,0.20
p2.3.g,97.50,95.00,3,0.18,110.00,110.00,1,0.18,120.00,120.00,1,0.21
p2.3.h,97.50,95.00,5,0.23,117.50,115.00,3,0.29,125.00,125.00,1,0.23
p2.3.i,110.00,110.00,1,0.17,140.00,140.00,1,0.21,158.75,155.00,3,0.64
p2.3.j,130.00,130.00,1,0.18,155.00,155.00,1,0.25,166.67,165.00,3,0.56
p2.3.k,150.00,150.00,1,0.21,200.00,200.00,1,0.34,200.00,200.00,1,0.32
p2.4.d,0.00,0.00,1,0.14,0.00,0.00,1,0.13,10.00,10.00,1,0.15
p2.4.e,25.00,25.00,1,0.14,25.00,25.00,1,0.14,35.00,35.00,1,0.15
p2.4.f,50.00,50.00,1,0.14,50.00,50.00,1,0.15,70.00,70.00,1,0.16
p2.4.g,50.00,50.00,1,0.14,70.00,70.00,1,0.16,70.00,70.00,1,0.16
p2.4.h,50.00,50.00,1,0.14,90.00,90.00,1,0.16,90.00,90.00,1,0.17
p2.4.i,70.00,70.00,1,0.15,105.00,105.00,1,0.16,105.00,105.00,1,0.18
p2.4.j,105.00,105.00,1,0.16,105.00,105.00,1,0.17,105.00,105.00,1,0.20
p2.4.k,112.50,110.00,5,0.22,120.00,120.00,1,0.19,120.00,120.00,1,0.23
\end{filecontents*}
\csvreader[
    before reading=\scriptsize\sisetup{round-mode=places, round-precision=2},
    after reading=\normalsize,
    longtable=|c|rrrr|rrrr|rrrr|,
    table head = \caption{Branch-and-price algorithm results for $21$ target instances.\label{tab:21}}\\
    \hline \textbf{n} & \multicolumn{4}{c|}{$|\Theta|=2$} & \multicolumn{4}{c|}{$|\Theta|=4$} & \multicolumn{4}{c|}{$|\Theta|=6$} \\ 
    \hline
    & \textbf{rub} & \textbf{opt} & \textbf{nn} & \textbf{time} & \textbf{rub} & \textbf{opt} & \textbf{nn} & \textbf{time} & \textbf{rub} & \textbf{opt} & \textbf{nn} & \textbf{time} \\ \hline\endfirsthead
    \caption{Branch-and-price algorithm results for $21$ target instances (continued).}\\
    \hline \textbf{n} & \multicolumn{4}{c|}{$|\Theta|=2$} & \multicolumn{4}{c|}{$|\Theta|=4$} & \multicolumn{4}{c|}{$|\Theta|=6$} \\ 
    \hline
    & \textbf{rub} & \textbf{opt} & \textbf{nn} & \textbf{time} & \textbf{rub} & \textbf{opt} & \textbf{nn} & \textbf{time} & \textbf{rub} & \textbf{opt} & \textbf{nn} & \textbf{time} \\
    \hline\endhead 
    \hline\endfoot,
    late after line=\\,
    ]{tables/full_21.csv}{1=\n,2=\rubtwo,3=\opttwo,4=\nntwo,5=\ttwo,6=\rubfour,7=\optfour,8=\nnfour,9=\tfour,10=\rubsix,11=\optsix,12=\nnsix,13=\tsix}%
    {\n & \rubtwo & \opttwo & \nntwo & \ttwo & \rubfour & \optfour & \nnfour & \tfour & \rubsix & \optsix & \nnsix & \tsix}

\begin{filecontents*}{tables/full_32.csv}
instance_name,rub_2,opt_2,nodes_2,time_2,rub_4,opt_4,nodes_4,time_4,rub_6,opt_6,nodes_6,time_6
p1.2.c,5.00,5.00,1,0.14,5.00,5.00,1,0.15,15.00,15.00,1,0.16
p1.2.d,15.00,15.00,1,0.15,20.00,20.00,1,0.17,25.00,25.00,1,0.17
p1.2.e,30.00,30.00,1,0.16,35.00,35.00,1,0.19,40.00,40.00,1,0.21
p1.2.f,40.00,40.00,1,0.19,55.00,55.00,1,0.26,60.00,60.00,1,0.35
p1.2.g,60.00,60.00,1,0.24,75.00,75.00,1,0.54,75.00,75.00,1,0.78
p1.2.h,75.00,75.00,1,0.54,87.50,85.00,7,1.95,95.00,95.00,1,1.41
p1.2.i,90.00,90.00,1,0.34,108.00,105.00,5,5.27,120.00,120.00,1,1.43
p1.2.j,95.00,95.00,1,0.40,125.00,125.00,1,1.51,140.00,140.00,1,8.04
p1.2.k,115.00,115.00,1,0.53,146.67,145.00,3,8.35,155.00,155.00,1,7.69
p1.2.l,130.00,130.00,1,0.81,160.00,160.00,1,4.12,170.00,170.00,1,19.69
p1.2.m,150.00,150.00,1,0.77,180.00,180.00,1,9.16,188.33,185.00,9,265.87
p1.2.n,160.00,160.00,1,1.43,190.00,190.00,1,7.86,207.00,205.00,3,489.63
p1.2.o,170.00,170.00,1,1.84,205.00,205.00,1,21.59,215.00,215.00,1,111.37
p1.2.p,173.33,170.00,5,9.26,210.00,210.00,1,16.71,222.50,220.00,7,1983.06
p1.2.q,187.50,185.00,3,9.24,220.00,220.00,1,29.79,234.23,230.00,5,3600.00
p1.2.r,200.00,200.00,1,5.64,230.00,230.00,1,117.22,245.00,245.00,11,3600.00
p1.3.d,0.00,0.00,1,0.14,5.00,5.00,1,0.15,15.00,15.00,1,0.15
p1.3.e,15.00,15.00,1,0.15,20.00,20.00,1,0.15,20.00,20.00,1,0.17
p1.3.f,20.00,20.00,1,0.14,25.00,25.00,1,0.16,30.00,30.00,1,0.17
p1.3.g,30.00,30.00,1,0.17,45.00,45.00,1,0.20,45.00,45.00,1,0.21
p1.3.h,40.00,40.00,1,0.16,50.00,50.00,1,0.21,60.00,60.00,1,0.23
p1.3.i,60.00,60.00,1,0.19,85.00,85.00,1,0.28,85.00,85.00,1,0.35
p1.3.j,75.00,75.00,1,0.20,95.00,95.00,1,0.33,100.00,100.00,1,0.62
p1.3.k,85.00,85.00,1,0.24,112.50,110.00,3,0.98,120.00,120.00,1,1.39
p1.3.l,105.00,105.00,1,0.27,122.50,120.00,9,2.81,135.63,135.00,3,4.15
p1.3.m,115.00,115.00,1,0.34,139.00,135.00,25,11.31,150.00,150.00,1,4.28
p1.3.n,130.00,130.00,1,0.57,155.00,155.00,1,1.55,170.00,170.00,1,3.25
p1.3.o,130.00,130.00,1,0.49,175.00,175.00,1,1.26,186.25,185.00,3,12.11
p1.3.p,135.00,135.00,1,0.47,177.50,175.00,9,10.40,195.00,195.00,1,5.79
p1.3.q,155.00,155.00,1,0.62,190.00,190.00,1,4.17,214.00,210.00,17,168.11
p1.3.r,170.00,170.00,1,1.53,213.33,210.00,15,29.65,230.00,230.00,1,26.24
p1.4.f,5.00,5.00,1,0.15,5.00,5.00,1,0.15,15.00,15.00,1,0.16
p1.4.g,15.00,15.00,1,0.16,25.00,25.00,1,0.16,25.00,25.00,1,0.17
p1.4.h,20.00,20.00,1,0.15,30.00,30.00,1,0.15,35.00,35.00,1,0.18
p1.4.i,35.00,35.00,1,0.16,45.00,45.00,1,0.18,45.00,45.00,1,0.20
p1.4.j,40.00,40.00,1,0.16,55.00,55.00,1,0.19,57.50,55.00,3,0.27
p1.4.k,60.00,60.00,1,0.17,70.00,70.00,1,0.22,76.67,75.00,5,0.56
p1.4.l,70.00,70.00,1,0.19,95.00,95.00,1,0.25,97.50,95.00,5,0.75
p1.4.m,85.00,85.00,1,0.20,105.00,105.00,1,0.31,110.00,110.00,1,0.48
p1.4.n,100.00,100.00,1,0.25,125.00,125.00,1,0.51,130.00,130.00,1,1.26
p1.4.o,110.00,110.00,1,0.23,137.50,135.00,7,1.63,145.00,145.00,1,1.34
p1.4.p,122.50,120.00,9,0.68,145.00,145.00,1,0.69,160.00,160.00,1,1.37
p1.4.q,130.00,130.00,1,0.35,155.00,150.00,21,6.66,170.38,165.00,29,23.34
p1.4.r,140.00,140.00,1,0.37,166.82,165.00,17,15.01,181.43,180.00,11,23.37
\end{filecontents*}
\csvreader[
    before reading=\scriptsize\sisetup{round-mode=places, round-precision=2},
    after reading=\normalsize,
    longtable=|c|rrrr|rrrr|rrrr|,
    table head = \caption{Branch-and-price algorithm results for $32$ target instances.\label{tab:32}}\\
    \hline \textbf{n} & \multicolumn{4}{c|}{$|\Theta|=2$} & \multicolumn{4}{c|}{$|\Theta|=4$} & \multicolumn{4}{c|}{$|\Theta|=6$} \\ 
    \hline
    & \textbf{rub} & \textbf{opt} & \textbf{nn} & \textbf{time} & \textbf{rub} & \textbf{opt} & \textbf{nn} & \textbf{time} & \textbf{rub} & \textbf{opt} & \textbf{nn} & \textbf{time} \\ \hline\endfirsthead
    \caption{Branch-and-price algorithm results for $32$ target instances (continued).}\\
    \hline \textbf{n} & \multicolumn{4}{c|}{$|\Theta|=2$} & \multicolumn{4}{c|}{$|\Theta|=4$} & \multicolumn{4}{c|}{$|\Theta|=6$} \\ 
    \hline
    & \textbf{rub} & \textbf{opt} & \textbf{nn} & \textbf{time} & \textbf{rub} & \textbf{opt} & \textbf{nn} & \textbf{time} & \textbf{rub} & \textbf{opt} & \textbf{nn} & \textbf{time} \\
    \hline\endhead 
    \hline\endfoot,
    late after line=\\,
    ]{tables/full_32.csv}{1=\n,2=\rubtwo,3=\opttwo,4=\nntwo,5=\ttwo,6=\rubfour,7=\optfour,8=\nnfour,9=\tfour,10=\rubsix,11=\optsix,12=\nnsix,13=\tsix}%
    {\n & \rubtwo & \opttwo & \nntwo & \ttwo & \rubfour & \optfour & \nnfour & \tfour & \rubsix & \optsix & \nnsix & \tsix}

 \begin{filecontents*}{tables/full_33.csv}
instance_name,rub_2,opt_2,nodes_2,time_2,rub_4,opt_4,nodes_4,time_4,rub_6,opt_6,nodes_6,time_6
p3.2.a,0.00,0.00,1,0.13,60.00,60.00,1,0.15,70.00,70.00,1,0.16
p3.2.b,60.00,60.00,1,0.15,100.00,100.00,1,0.17,110.00,110.00,1,0.18
p3.2.c,120.00,120.00,1,0.17,140.00,140.00,1,0.22,170.00,170.00,1,0.32
p3.2.d,140.00,140.00,1,0.20,172.31,170.00,11,1.54,200.00,200.00,1,0.48
p3.2.e,180.00,180.00,1,0.23,210.00,210.00,1,0.51,240.00,240.00,1,1.17
p3.2.f,210.00,200.00,29,1.55,250.00,250.00,1,1.04,270.00,270.00,1,4.78
p3.2.g,250.00,250.00,1,0.51,305.00,300.00,9,24.97,338.67,330.00,9,222.61
p3.2.h,290.00,290.00,1,0.56,347.50,340.00,17,48.72,377.50,370.00,7,95.44
p3.2.i,330.00,320.00,9,3.67,400.00,390.00,7,105.15,425.56,420.00,7,317.50
p3.2.j,380.00,380.00,1,1.58,446.67,430.00,51,642.51,481.67,480.00,5,1150.59
p3.2.k,412.50,400.00,61,38.44,493.33,490.00,5,267.78,521.25,510.00,11,3600.00
p3.2.l,455.00,440.00,15,22.57,535.00,520.00,23,1418.50,570.00,550.00,7,3600.00
p3.2.m,486.67,480.00,3,28.60,565.00,550.00,21,3600.00,600.00,600.00,1,214.79
p3.2.n,530.00,530.00,1,7.12,590.00,580.00,19,1963.75,625.00,620.00,3,3600.00
p3.2.o,550.00,550.00,1,112.08,622.50,610.00,17,3600.00,657.78,640.00,3,3600.00
p3.2.p,570.00,570.00,1,24.45,650.00,640.00,21,3600.00,690.00,670.00,5,3600.00
p3.2.q,600.00,600.00,1,94.89,680.00,670.00,9,3600.00,720.00,700.00,5,3600.00
p3.2.r,624.00,620.00,7,2527.55,707.14,700.00,3,3600.00,747.50,730.00,3,3600.00
p3.2.s,660.00,660.00,1,2760.56,738.33,730.00,3,3600.00,775.04,760.00,3,3600.00
p3.2.t,675.00,670.00,3,3600.00,761.25,750.00,3,3600.00,792.22,780.00,3,3600.00
p3.3.b,0.00,0.00,1,0.14,10.00,10.00,1,0.15,10.00,10.00,1,0.14
p3.3.c,10.00,10.00,1,0.15,80.00,80.00,1,0.16,90.00,90.00,1,0.18
p3.3.d,70.00,70.00,1,0.15,110.00,110.00,1,0.17,120.00,120.00,1,0.18
p3.3.e,140.00,140.00,1,0.17,160.00,160.00,1,0.21,190.00,190.00,1,0.27
p3.3.f,150.00,150.00,1,0.17,173.33,170.00,7,0.58,206.67,200.00,9,1.40
p3.3.g,170.00,170.00,1,0.19,212.00,210.00,3,0.72,240.00,240.00,1,0.53
p3.3.h,205.00,200.00,13,0.61,252.50,250.00,7,1.75,280.00,280.00,1,0.77
p3.3.i,230.00,230.00,1,0.23,290.00,290.00,1,0.95,320.00,320.00,1,2.03
p3.3.j,270.00,270.00,1,0.35,335.00,330.00,3,3.14,355.00,350.00,15,54.57
p3.3.k,300.00,300.00,1,0.41,390.00,380.00,41,20.95,410.00,410.00,1,11.37
p3.3.l,340.00,340.00,9,1.78,430.00,430.00,1,5.49,447.50,440.00,9,193.54
p3.3.m,380.00,380.00,1,0.65,465.00,460.00,19,79.81,490.80,490.00,3,154.73
p3.3.n,420.00,410.00,9,4.47,503.33,490.00,43,211.93,526.94,510.00,47,2742.33
p3.3.o,456.00,450.00,7,3.95,543.00,530.00,37,268.56,571.43,560.00,17,840.92
p3.3.p,500.00,500.00,1,1.27,585.37,560.00,682,3600.00,615.47,610.00,11,1408.70
p3.3.q,528.33,520.00,33,30.17,625.88,620.00,19,1124.98,654.73,640.00,41,3600.00
p3.3.r,553.26,540.00,83,191.60,653.71,640.00,87,3600.00,686.64,670.00,11,3600.00
p3.3.s,595.00,580.00,271,238.14,685.17,680.00,11,2079.71,713.82,710.00,5,3600.00
p3.3.t,621.67,610.00,171,454.72,704.75,700.00,7,1639.46,737.62,720.00,15,3600.00
p3.4.c,0.00,0.00,1,0.14,10.00,10.00,1,0.15,10.00,10.00,1,0.15
p3.4.d,0.00,0.00,1,0.14,60.00,60.00,1,0.15,80.00,80.00,1,0.15
p3.4.e,60.00,60.00,1,0.15,80.00,80.00,1,0.41,100.00,100.00,1,0.21
p3.4.f,80.00,80.00,1,0.15,120.00,120.00,1,0.18,130.00,130.00,1,0.18
p3.4.g,120.00,120.00,1,0.16,170.00,170.00,1,0.20,200.00,200.00,1,0.25
p3.4.h,170.00,170.00,1,0.18,180.00,180.00,1,0.23,210.00,210.00,1,0.28
p3.4.i,170.00,170.00,1,0.18,220.00,220.00,1,0.24,240.00,240.00,1,0.36
p3.4.j,200.00,200.00,1,0.19,240.00,240.00,1,0.30,260.00,260.00,1,0.42
p3.4.k,235.00,230.00,13,0.65,280.00,280.00,1,0.41,310.00,300.00,13,2.81
p3.4.l,265.00,260.00,5,0.50,300.00,300.00,1,0.35,320.00,320.00,1,0.77
p3.4.m,300.00,300.00,1,0.27,356.67,350.00,13,4.87,380.00,380.00,1,3.03
p3.4.n,320.00,320.00,1,0.65,396.67,390.00,23,15.33,416.67,410.00,15,61.58
p3.4.o,355.00,350.00,3,0.82,436.25,430.00,5,5.38,470.00,470.00,1,10.68
p3.4.p,390.00,390.00,1,0.40,496.67,490.00,13,33.15,520.00,520.00,1,5.24
p3.4.q,440.00,440.00,1,0.56,525.00,520.00,11,45.24,555.00,550.00,7,104.67
p3.4.r,460.00,460.00,1,0.73,548.57,530.00,67,353.35,570.00,560.00,11,358.29
p3.4.s,493.00,490.00,5,3.55,572.58,570.00,5,59.53,600.00,600.00,1,184.63
p3.4.t,522.50,510.00,19,7.70,630.00,630.00,1,18.33,667.50,660.00,35,1261.02
\end{filecontents*}
\csvreader[
    before reading=\scriptsize\sisetup{round-mode=places, round-precision=2},
    after reading=\normalsize,
    longtable=|c|rrrr|rrrr|rrrr|,
    table head = \caption{Branch-and-price algorithm results for $33$ target instances.\label{tab:33}}\\
    \hline \textbf{n} & \multicolumn{4}{c|}{$|\Theta|=2$} & \multicolumn{4}{c|}{$|\Theta|=4$} & \multicolumn{4}{c|}{$|\Theta|=6$} \\ 
    \hline
    & \textbf{rub} & \textbf{opt} & \textbf{nn} & \textbf{time} & \textbf{rub} & \textbf{opt} & \textbf{nn} & \textbf{time} & \textbf{rub} & \textbf{opt} & \textbf{nn} & \textbf{time} \\ \hline\endfirsthead
    \caption{Branch-and-price algorithm results for $33$ target instances (continued).}\\
    \hline \textbf{n} & \multicolumn{4}{c|}{$|\Theta|=2$} & \multicolumn{4}{c|}{$|\Theta|=4$} & \multicolumn{4}{c|}{$|\Theta|=6$} \\ 
    \hline
    & \textbf{rub} & \textbf{opt} & \textbf{nn} & \textbf{time} & \textbf{rub} & \textbf{opt} & \textbf{nn} & \textbf{time} & \textbf{rub} & \textbf{opt} & \textbf{nn} & \textbf{time} \\
    \hline\endhead
    \hline\endfoot,
    late after line=\\,
    ]{tables/full_33.csv}{1=\n,2=\rubtwo,3=\opttwo,4=\nntwo,5=\ttwo,6=\rubfour,7=\optfour,8=\nnfour,9=\tfour,10=\rubsix,11=\optsix,12=\nnsix,13=\tsix}%
    {\n & \rubtwo & \opttwo & \nntwo & \ttwo & \rubfour & \optfour & \nnfour & \tfour & \rubsix & \optsix & \nnsix & \tsix}
    
 \begin{filecontents*}{tables/full_64.csv}
instance_name,rub_2,opt_2,nodes_2,time_2,rub_4,opt_4,nodes_4,time_4,rub_6,opt_6,nodes_6,time_6
p6.2.d,0.00,0.00,1,0.14,168.00,168.00,1,0.23,180.00,180.00,1,0.35
p6.2.e,96.00,96.00,1,0.28,228.00,228.00,1,1.18,348.00,348.00,1,4.58
p6.2.f,192.00,192.00,1,0.67,288.00,288.00,1,3.11,504.00,504.00,1,53.91
p6.2.g,288.00,288.00,1,1.40,348.00,348.00,1,11.54,516.00,516.00,1,123.04
p6.2.h,348.00,348.00,1,1.48,420.00,420.00,1,16.54,588.00,588.00,1,190.88
p6.2.i,408.00,408.00,1,2.54,480.00,480.00,1,23.23,660.00,660.00,1,903.48
p6.2.j,480.00,480.00,1,3.91,540.00,540.00,1,49.76,--,744.00,1,3600.00
p6.2.k,552.00,552.00,1,8.98,612.00,612.00,1,42.52,852.00,840.00,5,3600.00
p6.2.l,606.00,606.00,1,15.78,660.00,660.00,1,140.96,948.00,936.00,3,3600.00
p6.2.m,660.00,660.00,1,28.34,720.00,714.00,7,3600.00,1003.44,984.00,3,3600.00
p6.2.n,726.00,726.00,1,16.80,787.20,774.00,3,3600.00,1063.63,1032.00,3,3600.00
p6.3.g,0.00,0.00,1,0.15,240.00,240.00,1,0.25,262.00,258.00,5,1.05
p6.3.h,72.00,72.00,1,0.22,324.00,318.00,5,2.52,432.00,420.00,9311,3600.00
p6.3.i,156.00,156.00,1,0.45,378.00,378.00,1,2.23,606.00,600.00,9,43.52
p6.3.j,282.00,282.00,1,0.75,432.00,432.00,1,3.79,696.00,690.00,19,442.52
p6.3.k,384.00,384.00,1,1.15,486.00,486.00,1,8.33,722.00,720.00,5,573.46
p6.3.l,450.00,450.00,1,1.36,558.00,558.00,1,23.80,776.00,768.00,7,746.36
p6.3.m,510.00,510.00,1,2.13,612.00,612.00,1,42.64,876.00,858.00,91,3600.00
p6.3.n,564.00,564.00,1,2.29,678.00,678.00,1,22.61,942.00,918.00,13,3600.00
p6.4.j,0.00,0.00,1,0.14,312.00,312.00,3,0.43,336.00,336.00,1,0.57
p6.4.k,0.00,0.00,1,0.14,384.00,384.00,1,0.65,511.50,498.00,6462,3600.00
p6.4.l,168.00,168.00,1,0.29,444.00,444.00,1,1.68,678.00,648.00,2173,3600.00
p6.4.m,240.00,240.00,1,0.44,528.00,528.00,1,2.82,804.00,780.00,39,463.02
p6.4.n,372.00,372.00,1,0.79,576.00,576.00,1,5.68,870.00,864.00,21,605.08
\end{filecontents*}
\csvreader[
    before reading=\scriptsize\sisetup{round-mode=places, round-precision=2},
    after reading=\normalsize,
    longtable=|c|rrrr|rrrr|rrrr|,
    table head = \caption{Branch-and-price algorithm results for $64$ target instances.\label{tab:64}}\\
    \hline \textbf{n} & \multicolumn{4}{c|}{$|\Theta|=2$} & \multicolumn{4}{c|}{$|\Theta|=4$} & \multicolumn{4}{c|}{$|\Theta|=6$} \\ 
    \hline
    & \textbf{rub} & \textbf{opt} & \textbf{nn} & \textbf{time} & \textbf{rub} & \textbf{opt} & \textbf{nn} & \textbf{time} & \textbf{rub} & \textbf{opt} & \textbf{nn} & \textbf{time} \\ \hline\endfirsthead
    \caption{Branch-and-price algorithm results for $64$ target instances (continued).}\\
    \hline \textbf{n} & \multicolumn{4}{c|}{$|\Theta|=2$} & \multicolumn{4}{c|}{$|\Theta|=4$} & \multicolumn{4}{c|}{$|\Theta|=6$} \\ 
    \hline
    & \textbf{rub} & \textbf{opt} & \textbf{nn} & \textbf{time} & \textbf{rub} & \textbf{opt} & \textbf{nn} & \textbf{time} & \textbf{rub} & \textbf{opt} & \textbf{nn} & \textbf{time} \\
    \hline\endhead
    \hline\endfoot,
    late after line=\\,
    ]{tables/full_64.csv}{1=\n,2=\rubtwo,3=\opttwo,4=\nntwo,5=\ttwo,6=\rubfour,7=\optfour,8=\nnfour,9=\tfour,10=\rubsix,11=\optsix,12=\nnsix,13=\tsix}%
    {\n & \rubtwo & \opttwo & \nntwo & \ttwo & \rubfour & \optfour & \nnfour & \tfour & \rubsix & \optsix & \nnsix & \tsix}
    
\begin{filecontents*}{tables/full_66.csv}
instance_name,rub_2,opt_2,nodes_2,time_2,rub_4,opt_4,nodes_4,time_4,rub_6,opt_6,nodes_6,time_6
p5.2.c,0.00,0.00,1,0.15,20.00,20.00,1,0.17,20.00,20.00,1,0.17
p5.2.d,60.00,60.00,1,0.17,70.00,70.00,1,0.19,80.00,80.00,1,0.26
p5.2.e,60.00,60.00,1,0.20,100.00,100.00,1,0.28,120.00,120.00,1,0.71
p5.2.f,160.00,160.00,1,0.26,170.00,170.00,1,0.70,190.00,190.00,1,2.38
p5.2.g,270.00,270.00,1,0.49,300.00,300.00,1,2.34,320.00,320.00,1,13.49
p5.2.h,300.00,300.00,1,0.76,330.00,330.00,1,6.56,340.00,340.00,1,32.68
p5.2.i,360.00,360.00,1,0.89,370.00,370.00,1,7.85,400.00,400.00,1,912.14
p5.2.j,400.00,400.00,1,1.33,510.00,510.00,1,12.88,520.00,520.00,1,118.49
p5.2.k,480.00,480.00,1,1.61,580.00,580.00,1,21.98,590.00,590.00,1,541.70
p5.2.l,550.00,550.00,1,3.65,650.00,650.00,1,36.13,690.00,690.00,1,250.01
p5.2.m,640.00,640.00,1,4.07,720.00,720.00,1,57.65,770.00,770.00,1,359.16
p5.2.n,720.00,720.00,1,6.94,790.00,790.00,1,83.28,840.00,840.00,1,598.49
p5.2.o,810.00,810.00,1,9.79,930.00,930.00,1,148.93,930.00,930.00,1,799.88
p5.2.p,880.00,880.00,1,13.27,1000.00,1000.00,1,153.45,1010.00,1010.00,1,1090.72
p5.2.q,1000.00,1000.00,1,18.40,1036.67,1035.00,7,954.03,1090.00,1090.00,1,1389.22
p5.2.r,1050.00,1050.00,1,27.41,1120.00,1120.00,1,902.37,1160.00,1160.00,1,1915.54
p5.2.s,1140.00,1140.00,1,36.39,1170.00,1170.00,1,1454.47,1220.00,1215.00,3,3600.00
p5.2.t,1170.00,1170.00,1,169.88,1240.00,1230.00,3,3600.00,1243.64,1230.00,3,3600.00
p5.2.u,1240.00,1240.00,1,135.85,1290.00,1270.00,3,3600.00,1303.19,1295.00,3,3600.00
p5.2.v,1340.00,1340.00,1,94.54,1359.89,1315.00,3,3600.00,1331.52,1310.00,3,3600.00
p5.2.w,--,1340.00,1,3600.00,1410.00,1380.00,3,3600.00,1352.46,1300.00,3,3600.00
p5.2.x,--,1420.00,1,3600.00,1470.91,1435.00,3,3600.00,1418.02,1355.00,3,3600.00
p5.2.y,1500.00,1500.00,1,1346.57,1517.55,1475.00,3,3600.00,1371.87,1300.00,3,3600.00
p5.2.z,--,1500.00,1,3600.00,1567.20,1520.00,3,3600.00,1402.45,1250.00,3,3600.00
p5.3.d,0.00,0.00,1,0.14,0.00,0.00,1,0.14,20.00,20.00,1,0.15
p5.3.e,0.00,0.00,1,0.14,60.00,60.00,1,0.17,60.00,60.00,1,0.18
p5.3.f,75.00,75.00,1,0.18,100.00,100.00,1,0.21,110.00,110.00,1,0.24
p5.3.g,90.00,90.00,1,0.18,135.00,135.00,1,0.55,150.00,150.00,1,0.46
p5.3.h,135.00,135.00,1,0.21,220.00,220.00,1,0.36,260.00,260.00,1,1.22
p5.3.i,240.00,240.00,1,0.29,255.00,255.00,1,0.66,285.00,285.00,1,3.04
p5.3.j,240.00,240.00,1,0.72,300.00,300.00,1,2.06,380.00,380.00,1,9.43
p5.3.k,420.00,420.00,1,0.56,460.00,460.00,1,3.57,470.00,470.00,1,24.55
p5.3.l,450.00,450.00,1,0.95,495.00,495.00,1,9.14,510.00,510.00,1,40.75
p5.3.m,495.00,495.00,1,1.09,525.00,525.00,1,8.66,575.00,575.00,1,303.91
p5.3.n,555.00,555.00,1,1.67,660.00,660.00,1,21.72,680.00,680.00,1,102.80
p5.3.o,600.00,600.00,1,2.77,765.00,765.00,1,30.39,780.00,780.00,1,157.61
p5.3.p,685.00,680.00,31,21.27,825.00,825.00,1,41.60,870.00,870.00,1,257.03
p5.3.q,730.00,730.00,1,16.31,870.00,870.00,1,62.26,975.00,975.00,1,333.92
p5.3.r,825.00,825.00,1,4.90,975.00,960.00,49,1738.52,1020.00,1010.00,29,3600.00
p5.3.s,930.00,930.00,1,7.17,1030.00,1020.00,33,1366.82,1100.00,1085.00,21,3600.00
p5.3.t,975.00,975.00,1,8.01,1095.00,1090.00,3,3600.00,--,1155.00,1,3600.00
p5.3.u,1060.00,1060.00,1,11.17,1170.00,1165.00,17,3600.00,1225.00,1210.00,3,3600.00
p5.3.v,1160.00,1145.00,33,201.33,1247.50,1240.00,11,3600.00,--,1305.00,1,3600.00
p5.3.w,1220.00,1215.00,75,805.49,1315.00,1305.00,7,3600.00,1357.50,1350.00,3,3600.00
p5.4.f,0.00,0.00,1,0.14,20.00,20.00,1,0.16,20.00,20.00,1,0.15
p5.4.g,0.00,0.00,1,0.14,80.00,80.00,1,0.17,110.00,110.00,1,0.20
p5.4.h,90.00,90.00,1,0.17,130.00,130.00,1,0.19,140.00,140.00,1,0.25
p5.4.i,120.00,120.00,1,0.18,140.00,140.00,1,0.25,190.00,190.00,1,0.39
p5.4.j,120.00,120.00,1,0.20,200.00,200.00,1,0.37,240.00,240.00,1,0.82
p5.4.k,240.00,240.00,1,0.23,320.00,320.00,1,0.57,340.00,340.00,1,1.63
p5.4.l,320.00,320.00,1,0.29,340.00,340.00,1,0.71,380.00,380.00,1,2.83
p5.4.m,320.00,320.00,1,0.35,380.00,380.00,1,1.54,420.00,420.00,1,7.72
p5.4.n,450.00,450.00,1,0.47,570.00,570.00,1,2.68,590.00,590.00,1,15.73
p5.4.o,600.00,600.00,1,0.96,620.00,620.00,1,5.64,660.00,660.00,1,71.75
p5.4.p,600.00,600.00,1,1.22,660.00,660.00,1,8.68,680.00,680.00,1,67.64
p5.4.q,660.00,660.00,1,2.28,700.00,700.00,1,13.33,720.00,720.00,1,2421.51
p5.4.r,680.00,680.00,1,1.41,740.00,740.00,1,16.09,800.00,800.00,1,646.34
p5.4.s,740.00,740.00,1,11.00,900.00,900.00,1,28.50,900.00,900.00,1,185.51
p5.4.t,800.00,800.00,1,33.44,1020.00,1020.00,1,39.10,1040.00,1040.00,1,248.08
p5.4.u,860.00,860.00,1,5.35,1040.00,1040.00,1,58.41,1083.33,1080.00,15,3600.00
p5.4.v,960.00,960.00,1,5.68,1160.00,1160.00,1,66.42,1180.00,1180.00,1,767.58
p5.4.w,1020.00,1020.00,1,12.40,1190.00,1190.00,1,277.51,1320.00,1320.00,1,682.37
p5.4.x,1100.00,1100.00,1,7.18,1260.00,1240.00,5,3600.00,1326.67,1320.00,3,3600.00
p5.4.y,1160.00,1160.00,1,7.62,1320.00,1320.00,1,387.25,1391.67,1380.00,3,3600.00
p5.4.z,1245.00,1240.00,19,115.84,1373.33,1370.00,3,3600.00,1455.00,1440.00,3,3600.00
\end{filecontents*}
\csvreader[
    before reading=\scriptsize\sisetup{round-mode=places, round-precision=2},
    after reading=\normalsize,
    longtable=|c|rrrr|rrrr|rrrr|,
    table head = \caption{Branch-and-price algorithm results for $66$ target instances.\label{tab:66}}\\
    \hline \textbf{n} & \multicolumn{4}{c|}{$|\Theta|=2$} & \multicolumn{4}{c|}{$|\Theta|=4$} & \multicolumn{4}{c|}{$|\Theta|=6$} \\ 
    \hline
    & \textbf{rub} & \textbf{opt} & \textbf{nn} & \textbf{time} & \textbf{rub} & \textbf{opt} & \textbf{nn} & \textbf{time} & \textbf{rub} & \textbf{opt} & \textbf{nn} & \textbf{time} \\ \hline\endfirsthead
    \caption{Branch-and-price algorithm results for $64$ target instances (continued).}\\
    \hline \textbf{n} & \multicolumn{4}{c|}{$|\Theta|=2$} & \multicolumn{4}{c|}{$|\Theta|=4$} & \multicolumn{4}{c|}{$|\Theta|=6$} \\ 
    \hline
    & \textbf{rub} & \textbf{opt} & \textbf{nn} & \textbf{time} & \textbf{rub} & \textbf{opt} & \textbf{nn} & \textbf{time} & \textbf{rub} & \textbf{opt} & \textbf{nn} & \textbf{time} \\
    \hline\endhead
    \hline\endfoot,
    late after line=\\,
    ]{tables/full_66.csv}{1=\n,2=\rubtwo,3=\opttwo,4=\nntwo,5=\ttwo,6=\rubfour,7=\optfour,8=\nnfour,9=\tfour,10=\rubsix,11=\optsix,12=\nnsix,13=\tsix}%
    {\n & \rubtwo & \opttwo & \nntwo & \ttwo & \rubfour & \optfour & \nnfour & \tfour & \rubsix & \optsix & \nnsix & \tsix}
    
\subsection{Value of multi-threading} \label{subsec:multi-threading}
To show the computational impact of using multiple threads in the concurrent B\&P algorithm of the D-DTOP, we selected all instances that were solved to optimality within the computation time limit of 1 hour and with the number of nodes explored in the branch-and-bound tree is greater than $1$ in Tables \ref{tab:21}--\ref{tab:66}. The total number of such instances was $118$. The single-threaded concurrent implementation is equivalent to the sequential implementation and throughout the rest of the article, we shall refer to this implementation as the sequential B\&P algorithm. As for the multi-threaded implementation, we shall refer to it as the concurrent B\&P algorithm. The sequential B\&P algorithm was run on all of these $118$ instances and Table \ref{tab:concurrency} shows the results for all those instances that consumed more than $5$ seconds of computation time for the sequential B\&P approach; the number of instances that satisfy this criteria is $73$. In Table \ref{tab:concurrency}, the nomenclature is as follows: \textbf{n}: instance name, \textbf{seq-time}: computation time of the sequential B\&P algorithm in seconds, \textbf{concurrent-time}: computation time of the concurrent B\&P algorithm in seconds, \textbf{max-solvers}: the maximum number of nodes in the branch-and-bound tree that are processed simultaneously by the concurrent B\&P algorithm, \textbf{improvement}: the gain in computation time in percent provided by the concurrent B\&P approach over the sequential B\&P approach.  It is clear from Table \ref{tab:concurrency} that the multi-threaded version of the B\&P i.e., the concurrent version of the B\&P algorithm with multiple threads is a clear winner and should always be preferred to the sequential B\&P algorithm.

\begin{filecontents*}{tables/one_thread.csv}
instance_name,num_targets,num_discretizations,one_thread_time,concurrent_time,max_concurrent_solves,improvement_factor
p1.2.i,32,4,6.63,5.26,2,20.58
p1.2.k,32,4,11.55,8.35,2,27.71
p1.2.m,32,6,468.36,265.87,3,43.23
p1.2.n,32,6,529.92,489.63,2,7.60
p1.2.p,32,6,2514.11,1983.06,2,21.12
p1.2.p,32,2,12.40,9.26,2,25.33
p1.2.q,32,2,11.08,9.24,2,16.63
p1.3.l,32,4,5.51,2.81,4,49.01
p1.3.l,32,6,11.05,4.14,2,62.48
p1.3.m,32,4,28.89,11.31,6,60.84
p1.3.o,32,6,15.53,12.11,2,22.02
p1.3.p,32,4,18.03,10.40,3,42.34
p1.3.q,32,6,377.75,168.11,4,55.50
p1.3.r,32,4,65.63,29.65,4,54.83
p1.4.q,32,6,61.80,23.34,7,62.23
p1.4.q,32,4,14.39,6.66,4,53.70
p1.4.r,32,4,29.01,15.01,4,48.27
p1.4.r,32,6,39.26,23.37,4,40.47
p3.2.g,33,6,277.87,222.61,2,19.89
p3.2.g,33,4,34.31,24.97,2,27.22
p3.2.h,33,4,108.13,48.72,5,54.94
p3.2.h,33,6,209.57,95.44,4,54.46
p3.2.i,33,4,135.38,105.15,2,22.33
p3.2.i,33,6,413.06,317.50,2,23.13
p3.2.i,33,2,5.15,3.67,4,28.66
p3.2.j,33,4,1697.00,642.51,8,62.14
p3.2.j,33,6,1449.51,1150.59,3,20.62
p3.2.k,33,2,107.00,38.44,8,64.07
p3.2.k,33,4,293.20,267.78,2,8.67
p3.2.l,33,2,44.93,22.57,3,49.76
p3.2.l,33,4,2554.32,1418.50,5,44.47
p3.2.m,33,2,31.19,28.60,2,8.29
p3.2.r,33,2,2989.88,2527.55,2,15.46
p3.3.j,33,6,100.84,54.57,3,45.88
p3.3.k,33,4,63.89,20.95,8,67.21
p3.3.l,33,6,266.21,193.54,3,27.30
p3.3.l,33,2,8.87,1.78,5,79.91
p3.3.m,33,4,172.70,79.81,4,53.78
p3.3.m,33,6,184.60,154.73,2,16.18
p3.3.n,33,4,458.83,211.93,8,53.81
p3.3.n,33,2,6.57,4.47,3,31.94
p3.3.o,33,2,5.94,3.95,2,33.43
p3.3.o,33,4,1071.98,268.56,8,74.95
p3.3.o,33,6,1809.52,840.92,5,53.53
p3.3.p,33,6,2721.07,1408.70,4,48.23
p3.3.q,33,2,68.97,30.17,4,56.26
p3.3.r,33,2,495.61,191.60,8,61.34
p3.3.s,33,2,1510.53,238.13,8,84.24
p3.3.s,33,4,3017.14,2079.71,4,31.07
p3.3.t,33,2,2059.64,454.72,8,77.92
p3.3.t,33,4,2089.67,1639.46,2,21.54
p3.4.m,33,4,7.82,4.86,3,37.80
p3.4.n,33,6,100.65,61.58,3,38.82
p3.4.n,33,4,34.74,15.33,7,55.86
p3.4.o,33,4,7.29,5.38,2,26.18
p3.4.p,33,4,59.91,33.15,4,44.67
p3.4.q,33,6,158.26,104.67,2,33.87
p3.4.q,33,4,68.54,45.24,3,34.00
p3.4.r,33,6,635.72,358.29,3,43.64
p3.4.r,33,4,1187.40,353.35,8,70.24
p3.4.s,33,4,85.52,59.53,3,30.40
p3.4.t,33,2,16.27,7.70,5,52.65
p3.4.t,33,6,1977.74,1261.02,3,36.24
p5.2.q,66,4,1656.54,954.03,3,42.41
p5.3.p,66,2,80.32,21.27,8,73.52
p5.3.v,66,2,762.87,201.33,8,73.61
p5.4.z,66,2,306.81,115.83,5,62.24
p6.3.i,64,6,84.77,43.52,4,48.66
p6.3.j,64,6,1333.49,442.52,7,66.81
p6.3.k,64,6,967.75,573.46,2,40.74
p6.3.l,64,6,1797.94,746.36,4,58.49
p6.4.m,64,6,1135.99,463.02,7,59.24
p6.4.n,64,6,1614.84,605.08,6,62.53
\end{filecontents*}
\csvreader[
    before reading=\scriptsize\sisetup{round-mode=places, round-precision=2},
    after reading=\normalsize,
    longtable=|c|c|r|r|r|r|,
    table head = \caption{Value of concurrency\label{tab:concurrency}}\\
    \hline \textbf{n} & $|\Theta|$ & \textbf{seq-time} & \textbf{concurrent-time} & \textbf{max-solvers}
    & \textbf{improvement} \\ \hline\endfirsthead
    \caption{Value of concurrency (continued).}\\
    \hline \textbf{n} & $|\Theta|$ & \textbf{seq-time} & \textbf{concurrent-time} & \textbf{max-solvers}
    & \textbf{improvement} \\
    \hline\endhead
    \hline\endfoot,
    late after line=\\,
    ]{tables/one_thread.csv}{1=\n,2=\numtargets,3=\disc,4=\tseq,5=\tcon,6=\concsolves,7=\improvement,}%
    {\n & \disc & \tseq & \tcon & \concsolves & \improvement}

For each of these $73$ instances, the scatter plot in Fig. \ref{fig:concurrent_scatter} provides a visual of how much of a gain in computation time is provided by the concurrent B\&P algorithm. In the Fig. \ref{fig:concurrent_scatter}, the lower the point below the line, the greater is the gain in the computation time provided by the multi-threaded B\&P over the sequential (single-threaded) B\&P. 

\begin{figure}
    \centering
    \includegraphics[scale=0.7]{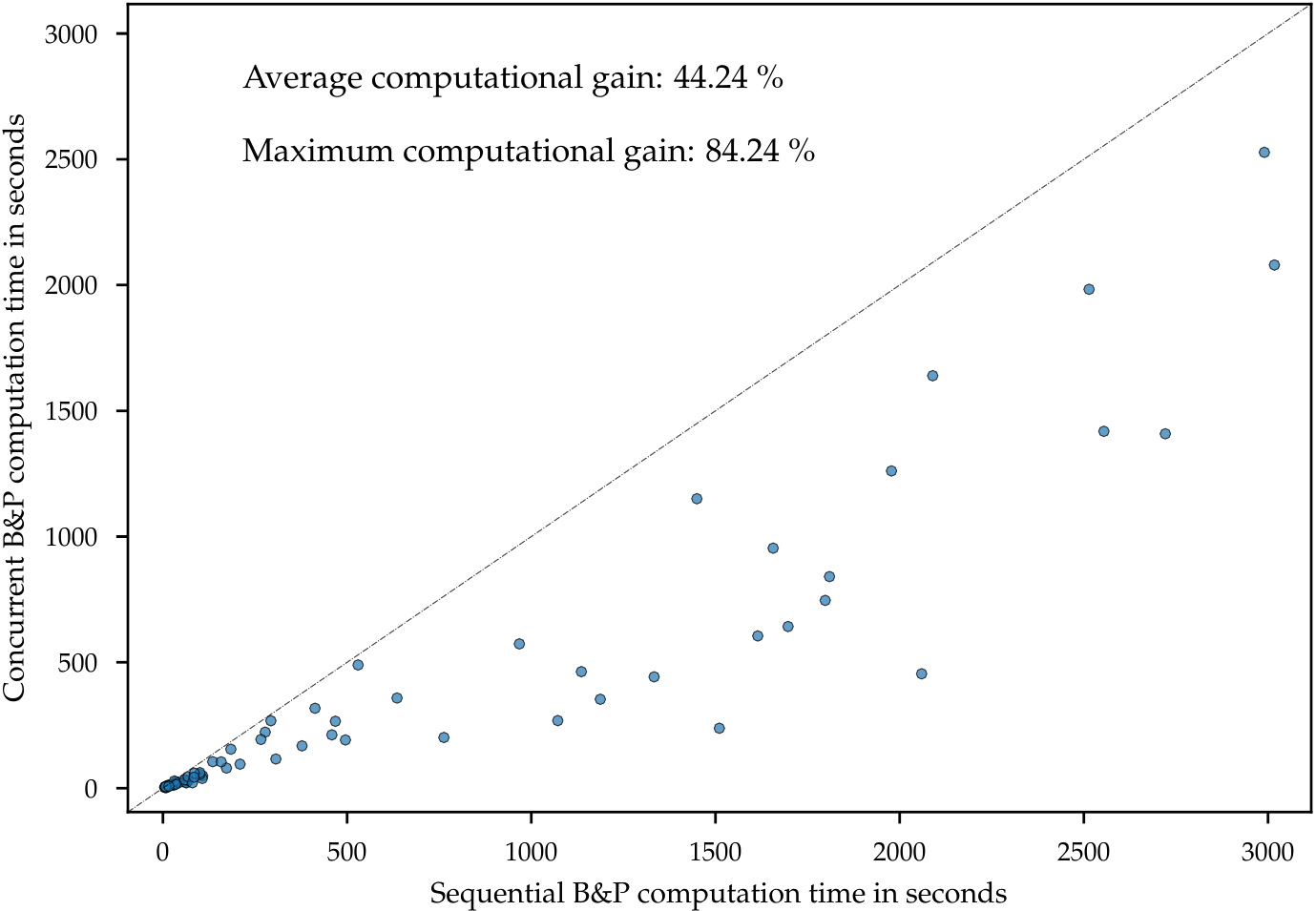}
    \caption{Scatter plot of the computation time to obtain the optimal solution when utilizing the multi-threaded implementation of the concurrent B\&P algorithm vs. the single-threaded implementation of the concurrent B\&P algorithm. The average and maximum gain in computation time is observed to be $45$\% and $85$\%, respectively. If a point lies on the line, then the computation time of the sequential B\&P algorithm is equal to that of the concurrent algorithm.}
    \label{fig:concurrent_scatter}
\end{figure}


\section{Conclusion and future work} \label{sec:conclusions}
This paper formulates a team orienteering problem for fixed-wing drones and presents a comprehensive B\&P algorithm to solve the problem to optimality. Finally, to the best of our knowledge, we also present the first ever concurrent implementation of the B\&P algorithm and show the computational gain over its sequential counterpart. The implementation for all experiments performed as a part of this research has been open-sourced and made available for use by the research community. Future work would focus on extending approaches presented in this paper to a wider class of vehicle routing problems with fixed-wing drones.

\section*{Acknowledgements}
Kaarthik Sundar acknowledges the funding provided by  LANL’s  Directed  Research  and  Development  (LDRD)  project: ``20200603ECR: Distributed Algorithms for Large-Scale Ordinary Differential/Partial Differential Equation (ODE/PDE) Constrained Optimization Problems on Graphs''. This work was carried out under the U.S. DOE Contract No. DE-AC52-06NA25396.

\bibliography{references}

\end{document}